\documentclass[11pt]{article}
\usepackage{latexsym}
\usepackage{mathrsfs}
\usepackage{amsmath,amsthm}
\usepackage{graphicx}
\usepackage{amssymb}
\usepackage{CJK}
\usepackage{color}
\newtheorem{thm}{Theorem}[section]
\newtheorem{cor}[thm]{Corollary}
\newtheorem{lem}[thm]{Lemma}
\newtheorem{prop}[thm]{Proposition}

\newtheorem{rem}[thm]{Remark}

\newtheorem{obs}{Observation}
\DeclareGraphicsExtensions{.eps,.eps.gz}
\normalsize \rm
\usepackage[centerlast]{caption}

\title{The fullerenes with a perfect star packing\footnote{
This work was supported in part by the National Natural Science Foundation of China (grant no. 11901458 and 11871256) and by the Fundamental Research Funds for the Central Universities (grant no. D5000200199).}}
\author{Ling-Juan Shi\thanks{Corresponding author.}}
\date{\small School of Software, Northwestern Polytechnical University, \\Xi'an, Shaanxi 710072, P. R. China\\
E-mails: shilj18@nwpu.edu.cn.}

\begin{document}
\maketitle
\begin{abstract}
 A spanning subgraph of a graph $G$ is called a perfect star packing in $G$ if every component of the spanning subgraph is isomorphic to the star graph $K_{1,3}$.
An efficient dominating set of graph $G$ is a vertex subset $D$ of $G$ such that each vertex of $G$ not in $D$ is adjacent to exactly one vertex from $D$ and any two vertices of $D$ are not adjacent in $G$.
Fullerene graph is a connected plane cubic graph with only pentagonal and hexagonal faces, which is the molecular graph of carbon fullerene.
Clearly, a perfect star packing in a fullerene graph $G$ on $n$ vertices will exist if and only if $G$ has an efficient dominating set of cardinality $\frac{n}{4}$. The problem of finding an efficient dominating set is algorithmically hard \cite{Alg_hard}. In this paper, we give a characterization for a fullerene graph to own a perfect star packing. And mainly
show that it is necessary for a fullerene $G$ owning a  perfect star packing to have order being divisible by $8$. This answers an open problem asked by Dosli\'{c} et. al. and also shows that a fullerene graph with an efficient dominating set has $8n$ vertices.
By the way, we find some counterexamples for the necessity of Theorem $14$ in \cite{Doslic} and list some forbidden configurations to preclude the existence of a perfect star packing of type $P0$.

\end{abstract}
\baselineskip=0.3in

\section{Introduction}
A chemical graph is a simple finite graph in which vertices denote the atoms and edges denote the chemical bonds in underlying chemical structure.
Perfect matchings of a chemical graph correspond to kekul\'{e} structures of the molecule, which feature in the calculation of molecular energies associated with benzenoid hydrocarbon molecules \cite{pm_molecular_energy}.
Alternating sextet faces (sextet patterns) also play a meaningful role in the prediction of molecular stability, in particular, but not only, in benzenoid compounds.
Although for fullerenes, the two structures do not play the same role as in benzenoid compounds, they received numerous attention in recent years, see  \cite{ clar_fullerene, pm_stability, resonant_hexagons_fullerene, pms_lowbound, maxiclar_fullerene, clarstructure_F, clar_upbound_F, pms_n_lowbound} etc.

A perfect matching in a graph $G$ may be viewed as a collection of subgraphs of $G$, each of which is isomorphic to $K_2$, whose vertex sets partition the vertex set of $G$. This is naturally generalized by replacing $K_2$ by an arbitrary graph $H$.
For a given graph $H$, an \emph{$H$-packing} of $G$ is the set of some vertex disjoint subgraphs, each of which is isomorphic to $H$. From the optimization point of view, the maximum $H$-packing problem is to find the maximum number of vertex disjoint copies of $H$ in $G$ called the \emph{packing number}. An $H$-packing in $G$ is called \emph{perfect} if it covers all the vertices of $G$.
If $H$ is isomorphic to $K_2$, the maximum (perfect) $H$-packing problem becomes the familiar maximum (perfect) matching problem.
If $H$ is the cycle $C_6$ of length $6$, for a fullerene or a hexagonal system $G$, the packing number is related to the Clar number (the maximum number of mutually disjoint sextet patterns) of $G$.
If $H$ is the star graph $K_{1,3}$, it is the maximum star packing problem. If a $K_{1,3}$-packing covers all the vertices of $G$, we call it being a \emph{perfect star packing}.
For a given family $\mathcal{F}$ of graphs, an $H$-packing concept can also be generalized to an $\mathcal{F}$-packing (we refer the reader to \cite{family_packing} for the definition).

Packing in graphs is an effective tool as it has lots of applications in applied sciences.
$H$-Packing, is of practical interest in the areas of scheduling \cite{scheduling}, wireless sensor tracking \cite{wirelesssensortracking}, wiring-board design, code optimization \cite{codeoptimization} and many others.
Packing problems were already studied for Carbon Nanotubes \cite{packing_carbon_nanotube}.
Packing lines in a hypercube have been studied in \cite{Packing_lines}. $H$-packing was determined for honeycomb \cite{family_packing} and hexagonal network \cite{hexagonal_network}. For representing chemical compounds or to problems of pattern recognition and image processing, $P_3$-packing has some applications in chemistry \cite{P3_packing}.
Dosli\'{c} et. al \cite{Doslic} have investigated which fullerene graphs allow perfect star packings, and have considered generalized fullerene graphs and packings of other graphs into classical and generalized fullerenes. They also listed several open problems.

In the following section we introduce necessary preliminaries and
characterize the classical fullerenes which have a perfect star packing. Section $3$
gives a negative answer to the problem ``Is there a fullerene on $8n+4$ vertices with a perfect star packing?" asked by Dosli\'{c} et. al \cite{Doslic}.
This implies that a fullerene graph with an efficient dominating set must has $8n$ vertices.
In section $4$, we generalize the Proposition $1$ in reference \cite{Doslic} and give three counterexamples for Theorem $14$  in the same paper.
And some forbidden configurations are listed to preclude the existence of a perfect star packing of type $P0$.

\section{Characterization of fullerenes with a perfect star packing}
A \emph{fullerene} graph (simply fullerene) is a cubic $3$-connected plane graph with only pentagonal and hexagonal faces.  It follows from the Euler formula that there must be exactly $12$ pentagons in every fullerene graph.
Such graphs are suitable models for carbon fullerene molecules: carbon atoms are
represented by vertices, whereas edges represent chemical bonds between two atoms (see \cite{fowler1995, LZ18}).
In a classical paper by Gr\"{u}nbaum and Motzkin \cite{Gru}, we know that a fullerene graph with $n$ vertices exists for all even $n\geq20$ except for $n=22$.
 Klein and Liu \cite{Klein_Liu} used a similar approach to show that there exist fullerene graphs on $n$ vertices with isolated pentagons for $n=60$ and for each even $n\geq70$.
 We refer the reader to the monograph \cite{fowler1995} for a systematic introduction on fullerene graphs.

A cycle of a fullerene graph $G$ is a \emph{facial cycle} if it is the boundary of a face in $G$, otherwise, it is a \emph{non-facial} cycle.
Clearly, each pentagon and hexagon in $G$ is a facial cycle since $G$ is $3$-connected and any $3$-edge-cut is trivial \cite{af_fullerene}.
In paper \cite{Doslic}, the authors obtained the following basic conclusions.
\begin{prop}[\cite{Doslic}]\label{1-pentagon}
Let $S$ be a perfect star packing of fullerene graph $G$. Then each pentagon of $G$ can contain at most one center of a star in $S$.
\end{prop}
\begin{lem}[\cite{Doslic}]\label{2-pentagons}
Let $S$ be a perfect star packing of fullerene graph $G$. Then a vertex shared by two pentagons of $G$ cannot be the center of a star in $S$.
\end{lem}

Recall that a vertex set $X$ of a graph $G$ is said to be \emph{independent} if any two vertices in $X$ are not adjacent in $G$. A cycle $C=v_1v_2\cdots v_kv_1$ in $G$ is called \emph{induced} if $v_i$ has only two adjacent vertices $v_{i+1}$ and $v_{i-1}$ around the $k$ vertices $v_1, v_2, \cdots, v_k$ (note that if $i=k, i+1:=1$ and if $i=1, i-1:=k$). Otherwise, there exists some $i$ and $j\notin \{i-1, i+1\}$ such that $v_i$ and $v_j$ are adjacent in $G$, the edge $v_iv_j$ is a \emph{chord} of $C$ and $C$ is not induced.
A subgraph $R$ of a graph $G$ is \emph{spanning} if $R$ covers all the vertices of $G$.
For a vertex $v$ of a graph $G$, we call vertex $u$ being a \emph{neighbor} of $v$ in $G$ if $u$ is adjacent to $v$ in $G$.
\begin{thm}\label{characterization}
Let $G$ be a fullerene graph.
Then $G$ has a perfect star packing if and only if $G$ has an independent vertex set $S^*$ such that each component of
 $G-S^*$ is an induced cycle in $G$.
\end{thm}
\begin{proof}
If $G$ has a perfect star packing $S$, then $S$ is a spanning subgraph of $G$ and any component in $S$ is isomorphic to a star graph $K_{1,3}$. Let $S^*$ be the set of all $3$-degree vertices in $S$. Clearly, $S^*$ is an independent vertex set in $G$ and any vertex in $G-S^*$ has degree $2$. So each component of $G-S^*$ is an induced cycle in $G$.

Let $S^*$ be an independent vertex set of $G$ such that each component of $G-S^*$ is an induced cycle in $G$.
Clearly, each vertex in $S^*$ and its three neighbors induce a star graph $K_{1,3}$.
We collect all these star graphs and denote this set by $\mathcal{H}$.
For any vertex $x$ on a cycle $C$ in $G-S^*$, $x$ has exactly one neighbor in $S^*$ since $G$ is $3$-regular and induced cycle $C$ is a component of $G-S^*$.
So $\mathcal{H}$ is a spanning subgraph of $G$ and each component of $\mathcal{H}$ is a star graph $K_{1, 3}$, that is, $\mathcal{H}$ is a perfect star packing of $G$.
\end{proof}
We note that star graph $K_{1,3}$ has exactly one \emph{center} (the vertex of degree $3$) and three leaves. A perfect star packing $S$ of a fullerene graph $G$ is a spanning subgraph of $G$ each component of which is a star graph $K_{1,3}$.
We call each $1$-degree vertex in $S$ being a \emph{leaf}.
In the following, we denote by $C(S)$ the set of all the centers of stars in $S$.
\begin{rem}\label{remark1}
Let $S$  be a perfect star packing of fullerene graph $G$. Then\\
$1.$ $C(S)$ is an independent vertex set in $G$.\\
$2.$ Any leaf in $S$ has exactly one neighbor belonging to $C(S)$ and has exactly two neighbors being leaves in $S$.\\
$3.$ Each cycle in $G-C(S)$ does not have a chord.
\end{rem}

\begin{prop}\label{1-hexagon}
Each hexagon can contain at most two centers of a perfect star packing of fullerene graph $G$. If a hexagon $h$ contains two such centers, then they are antipodal points on the hexagon $h$.
\end{prop}
\begin{proof}
Let $h$ be a hexagon in $G$. We denote the six vertices of $h$ by $v_1, v_2, \dots, v_6$ in the clockwise direction.
If vertex $v_1$ is the center of a star $H$ in a perfect star packing $S$ of $G$, then $v_2$ and $v_6$ are two leaves in $H$. Hence both $v_3$ and $v_5$ are leaves in $S$ by Remark \ref{remark1} 2.
Clearly, $v_4$ could be the center of a star in $S$.
Hence $h$ has exactly one center of $S$ or has exactly two centers of $S$ which are antipodal points on $h$.
\end{proof}

\section{The order of fullerenes with a perfect star packing}
To show the main conclusion, we need to prepare as follows.
\begin{figure}[htbp!]
\centering
\includegraphics[height=4.5cm]{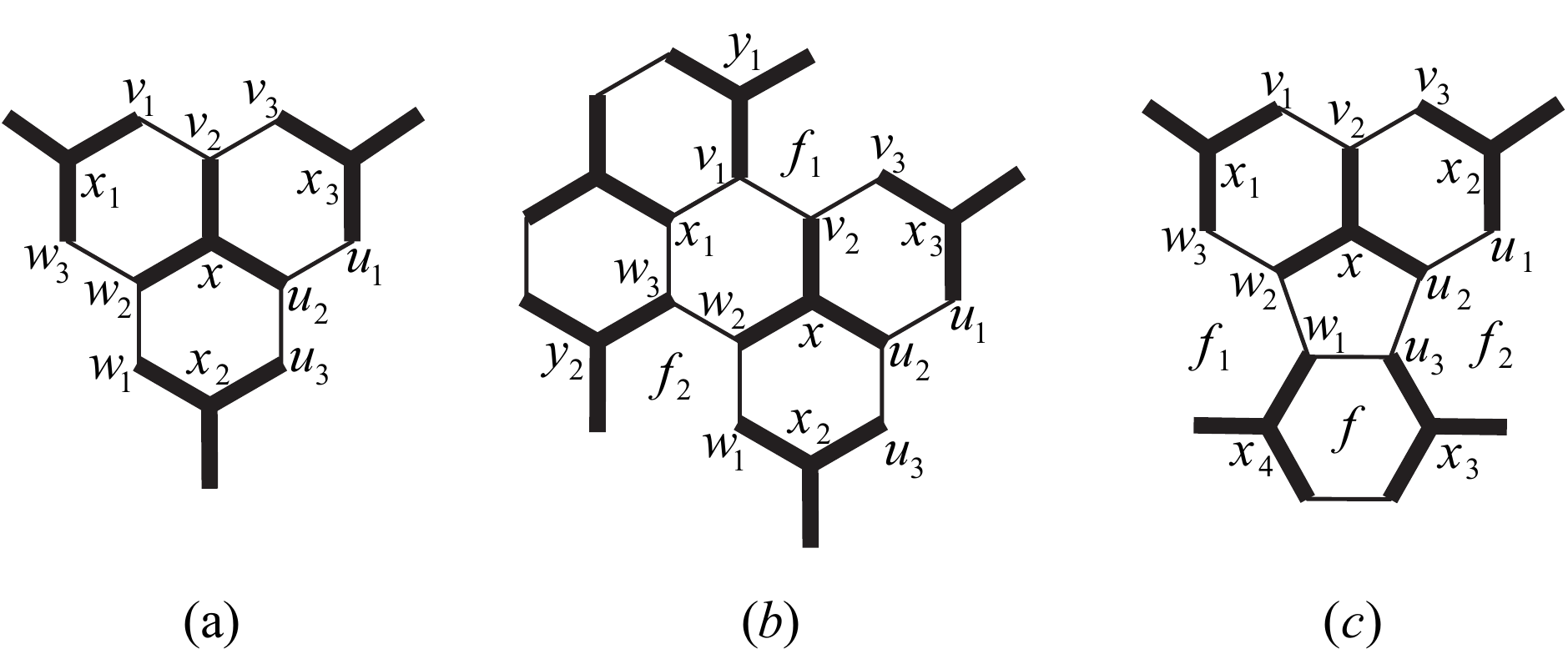}
\caption{\label{structure1}{\small (a) Type $1$; (b) Type $2$; (c) Type $3$.}}
\end{figure}
\begin{lem}\label{structuure-star}
Let $S$  be a perfect star packing of fullerene graph $G$. Then for any vertex $x\in C(S)$, all the vertices on the three faces sharing $x$ are covered by $S$ as Type $1$, Type $2$ or Type $3$ \emph{(}see Fig. \ref{structure1}\emph{)}.
\end{lem}
\begin{proof}
By the Lemma \ref{2-pentagons}, at most one of the three faces sharing $x$ is a pentagon since $x\in C(S)$. There are two cases as follows.

\textbf{Case 1.} The three faces sharing $x$ are all hexagons.

Clearly, $x$ has three antipodal points on the three hexagons sharing $x$, denoted by $x_1$, $x_2$ and $x_3$ respectively as depicted in Fig. \ref{structure1} (a).
By Remark \ref{remark1} 2, the two neighbors $v_1$ and $v_3$ of $v_2$ are leaves in $S$. Similarly, $u_1, u_3, w_1$ and $w_3$ are also leaves in $S$.
We claim that at least two of $x_1, x_2$ and $x_3$ are centers of stars in $S$.
If $x_1$ is not the center of a star in $S$, then $x_1$ is a leaf in $S$. So the third neighbor of $v_1$, say $y_1$, is the center of a star in $S$ (see Fig. \ref{structure1} (b)). Similarly, the third neighbor of $w_3$, say $y_2$, is also the center of a star in $S$.
Since the three vertices $v_1, v_2$ and $v_3$ are leaves in $S$ and $y_1\in C(S)$, the face $f_1$ has only one center of $S$ by Propositions \ref{1-hexagon} and \ref{1-pentagon}.
Hence the two neighbors of $v_3$ on $f_1$ are leaves.
By Remark \ref{remark1} 2, $x_3$ is the center of a star in $S$, that is,
$x_3\in C(S)$.
Similarly, $w_1$ is a leaf in $S$ and the two neighbors of $w_1$ on $f_2$ are all leaves in $S$. Hence $x_2\in C(S)$.
So at least two of $x_1$, $x_2$ and $x_3$ belong to $C(S)$.
If exactly two of $x_1$, $x_2$ and $x_3$ belong to $C(S)$, without loss of generality, we suppose that $x_2, x_3\in C(S)$, then all the vertices on the three faces sharing $x$ are covered by $S$ as Type $2$.
If all the three vertices $x_1$, $x_2$ and $x_3$ belong to $C(S)$ (see Fig. \ref{structure1} (a)), then all the vertices on the three faces sharing $x$ are covered by $S$ as Type $1$.

\textbf{Case 2.} Exactly one of the three faces sharing $x$ is a pentagon.

By Proposition \ref{1-pentagon}, $w_1$ and $u_3$ are leaves in $S$ (see Fig. \ref{structure1} (c)). Hence $x_4, x_3\in C(S)$ and $f$ is a hexagon by Remark \ref{remark1}. 2 and Proposition \ref{1-hexagon}.
By Remark \ref{remark1}. 2, the neighbor $w_3$ of $w_2$ is a leaf in $S$ since the neighbor $x$ of $w_2$ belong to $C(S)$.
Hence the other vertices on $f_1$ except for $x_4$ are all leaves in $S$ by Propositions \ref{1-pentagon} and \ref{1-hexagon}.
This follows that the neighbor $x_1$ of $w_3$ is the center of a star in $S$ by Remark \ref{remark1} 2.
Similarly, we can show $x_2\in C(S)$. Hence all the vertices on the three faces sharing $x$ are covered by $S$ as Type $3$ (see Fig. \ref{structure1} (c)).
\end{proof}

\begin{cor}\label{pentaon-and-nonfacialcycle}
Let $S$ be a perfect star packing of fullerene graph $G$. If a pentagon $P$ of $G$ has a vertex $x\in C(S)$, then $G-C(S)$ has a non-facial cycle $C$ of $G$ such that the path $P-x$ is a subgraph of $C$.
\end{cor}
\begin{proof}
By Proposition \ref{1-pentagon}, $x$ is shared by this pentagon $P$ and two hexagons. So all the vertices on the three faces sharing $x$ are covered by $S$ as Type $3$ (see Fig. \ref{structure1} (c)).
Clearly, the path $P-x$ is a subgraph of a cycle $C$ in $G-C(S)$ and $C$ is a non-facial cycle of $G$.
\end{proof}
We note that $3$-connected graphs have only one embedding up to equivalence \cite{Diestel_book}. If we embed a fullerene graph $G$ in the plane, then any non-facial cycle $C$ of $G$ as a Jordan curve separates the plane into two regions, denoted by $R_1^*$ and $R_2^*$, each of which has the entire $C$ as its frontier. We denote the subgraph of $G$ induced by the vertices lying in the interior of $R_i^*$  by $G_i$, $i=1, 2$. Here we note that $\{V(G_1), V(G_2), V(C)\}$ is a partition of all the vertices of $G$.
We say that
$C$ \emph{divide} the graph $G$ into two \emph{sides} $G_1$ and $G_2$.
\begin{thm}
Let $S$ be a perfect star packing of fullerene graph $G$ and $C$ be a cycle in $G-C(S)$. Then $C(S)$ does not have a vertex which has three neighbors on $C$.
\end{thm}
\begin{figure}[htbp!]
\centering
\includegraphics[height=4.8cm]{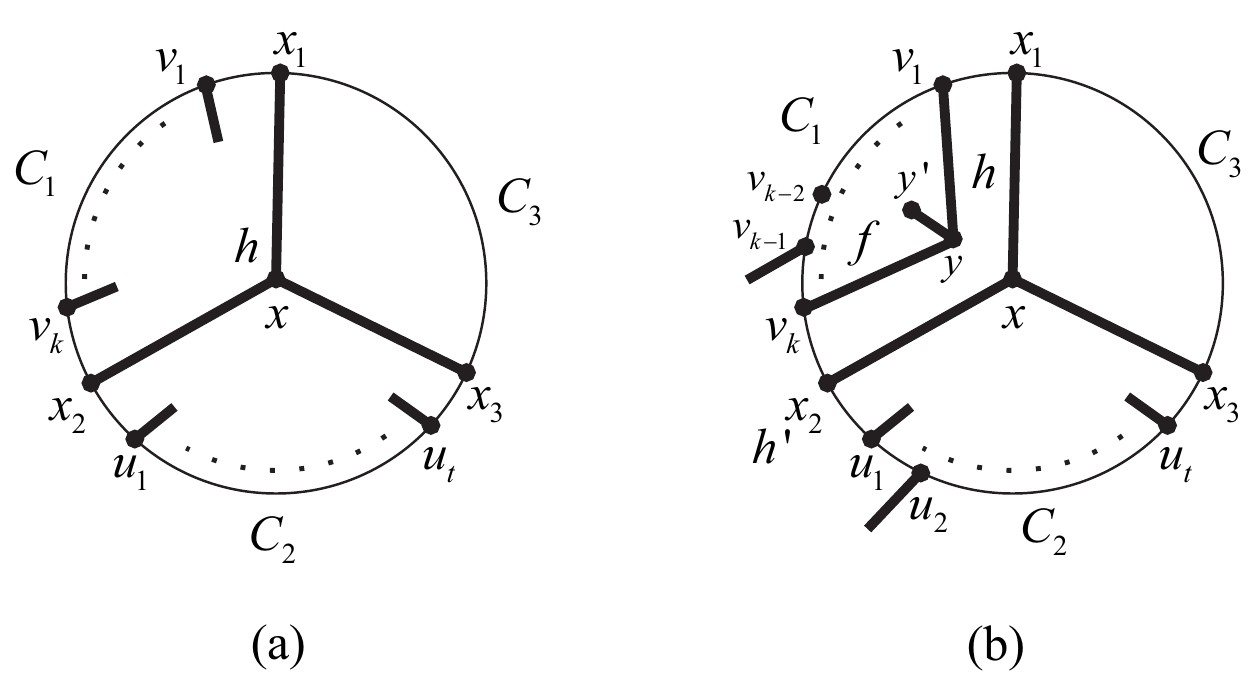}
\caption{\label{1center-3}{\small $x\in C(S)$ has three neighbors on $C$.}}
\end{figure}
\begin{proof}
If $C$ is a facial cycle of $G$, then $C$ is a pentagon or a hexagon. The conclusion clearly holds.
Now, let $C$ be a non-facial cycle of $G$.
Then $C$ divides $G$ into two sides, denoted by $H_1$ and $H_2$ respectively.
We note that all vertices on $C$ are leaves in $S$ since $C$ is a cycle in $G-C(S)$.
On the contrary, we suppose that there is a vertex $x\in C(S)$ which has three neighbors on $C$, denoted by $x_1, x_2$ and $x_3$ respectively. Without loss of generality, we suppose that $x\in V(H_1)$ (see Fig. \ref{1center-3} (a)).
The three vertices separate the circle $C$ into three sections, denoted by $C_1, C_2$ and $C_3$ respectively, each of which is a path with $x_i$ and $x_{i+1}$ as two terminal ends, $i=1, 2, 3$ (if $i=3$, then $i+1:=1$).
From Lemma \ref{structuure-star} we know that at most one of $x_1C_1x_2x$, $x_2C_2x_3x$ and $x_3C_3x_1x$ is a facial cycle of $G$ since $C$ is a cycle in $G-C(S)$.
Next, we suppose that $x_1C_1x_2x$ and $x_2C_2x_3x$ are non-facial cycles of $G$. Let $C_1=x_1v_1v_2\cdots v_kx_2$, $C_2=x_2u_1u_2\cdots u_tx_3$.
So $k\geq5$ and $t\geq5$ since any non-facial cycle of $G$ has length at least $8$.
By Remark \ref{remark1}.3,
$C$ does not have a chord.
So $v_1v_k\notin E(G)$ and $u_1u_t\notin E(G)$.
This implies that $h$ is a hexagon face of $G$, and $x_1, x, x_2$ and $v_1, v_k$ are five vertices on $h$. We denote the sixth vertex of $h$ by $y$. Clearly, $y\in V(H_1)$ by the planarity of $G$ (see Fig. \ref{1center-3} (b)).
Similarly, both $u_1$ and $u_t$ have a common neighbor in $H_1$.

Since $S$ is a perfect star packing of $G$ and the two neighbors $x_1$ and $v_2$ of $v_1$ are leaves in $S$,
$y$ is the center of a star in $S$.
If the third neighbor of $y$ is on $C$, then it is on $C_1$, denoted it by $v_r$. The three neighbors of $y$ separate the circle $C$ into three sections, two of which are subgraphs of $C_1$, denoted by $C_1^1$ and $C_1^2$ respectively. As the above discussion, we know that one of $v_1C_1^1v_ry$ and $v_rC_1^2v_ky$ is a non-facial cycle of $G$. By the recursive process and the finiteness of the order of $G$, we can suppose that the third neighbor of $y$ is not on $C$, and denoted it by $y'$.

See Fig. \ref{1center-3} (b), the five vertices $v_{k-1}, v_k, x_2, u_1, u_2$ belong to a common facial cycle $h'$ of $G$.
Since $C$ does not have a chord by Remark \ref{remark1} 3, $v_{k-1}$ and $u_2$ are not adjacent in $G$. So $h'$ is a hexagon. By the planarity of $G$, $v_{k-1}$ and $u_2$ have a common neighbor in $H_2$.
so $v_{k-2}, v_{k-1}, v_k, y$ and $y'$ are on a face of $G$, say $f$.
If $f$ is a pentagon, then $v_{k-2}$ is adjacent to $y'$.
So all the three neighbors of $v_{k-2}$ are leaves in $S$. This implies a contradiction since $v_{k-2}$ is also a leaf in $S$.
If $f$ is a hexagon, then $v_{k-2}$ and $y'$ have a common neighbor, denoted by $z$.
Clearly, $z$ is $v_{k-3}$ or not.
For $z=v_{k-3}$, the three neighbors of $v_{k-3}$ are all leaves in $S$, a contradiction.
For $z\neq v_{k-3}$, by Remark \ref{remark1} 2, $z$ is a leaf in $S$ since $y'$ has a neighbor $y\in C(S)$.
So the three neighbors of $v_{k-2}$ are all leaves in $S$, a contradiction.
All these contradictions imply that $C(S)$ does not have a vertex which has three neighbors on $C$.
\end{proof}

Let $S$ be a perfect star packing of fullerene graph $G$ and $C$ be a cycle in $G-C(S)$ which is a non-facial cycle of $G$. $C$ divides $G$ into two sides, denoted by $H_1$ and $H_2$ respectively.
Set $C^i$ be the set of all the vertices on $C$ each of which has a neighbor in $H_i$, $i=1, 2$.
Clearly, $\{C^1, C^2\}$ is a partition of $V(C)$.
$G[C^i]$ is a vertex induced subgraph of $G$ which has vertex set $C^i$ and any two vertices of $C^i$ are adjacent if and only if they are adjacent in $G$.
See Fig. \ref{structure-C-t}, $G[C^1]$ is depicted as red and $G[C^2]$ is depicted as blue.
In the following, we use these symbols no longer explaining.
\begin{lem}\label{structure-C}
For $i=1, 2$, if a vertex $x$ on $C$ has a neighbor in $H_i$, then the component of the induced subgraph $G[C^i]$ which contains $x$ is a path with $2$ or $3$ vertices.
\end{lem}
\begin{proof}
We suppose that $x$ on $C$ has adjacent vertex in $H_1$. For the convenience of the following description, set $C:=xv_1v_2\cdots v_kx$.
Since $C$ is a cycle in $G-C(S)$ which is a non-facial cycle of $G$, the length of $C$ is at least $8$.
So $k\geq7$.
There are three cases for the two neighbors $v_1$ and $v_k$ of $x$ on $C$.

\textbf{Case 1.} Both $v_1$ and $v_k$ have neighbors in $H_2$.

In this case, the three vertices $v_1, x$ and $v_k$
lie on the same face $f$ of $G$ (see Fig. \ref{structure-C-beside} (a)).
Since all the vertices on $C$ are leaves in $S$, the other neighbor of $v_1$ (resp. $v_k$) which is not on $C$ is the center of a star in $S$. So $f$ has two vertices in $C(S)$ which are the centers of two stars in $S$ covered $v_1$ and $v_k$, respectively. So $f$ is a hexagon by Proposition \ref{1-pentagon}.
But the case cannot hold by Propositions \ref{1-hexagon}.
\begin{figure}[htbp!]
\centering
\includegraphics[height=5.5cm]{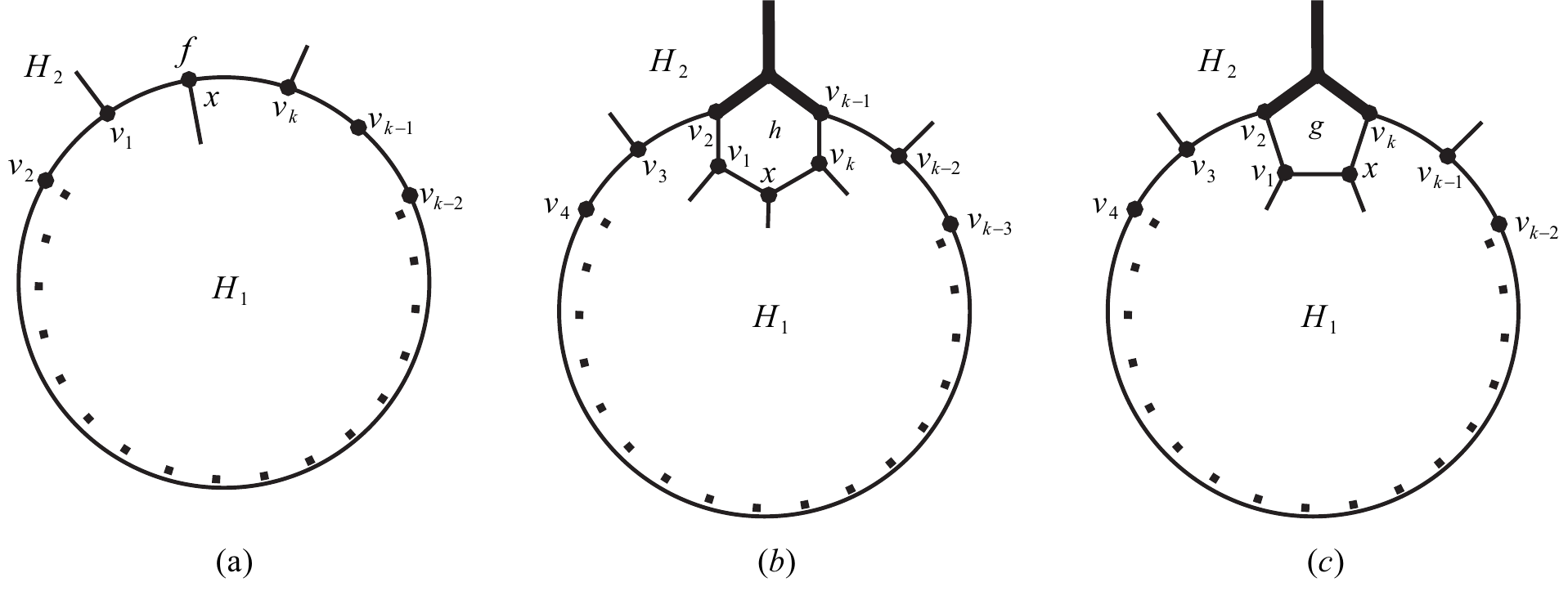}
\caption{\label{structure-C-beside}{\small (a) $v_1$ and $v_k$ have neighbors in $H_2$; (b) $v_1$ and $v_k$ have neighbors in $H_1$; (c) $v_1$ has a neighbor in $H_1$, $v_k$ has one in $H_2$.}}
\end{figure}

\textbf{Case 2.} Both $v_1$ and $v_k$ have neighbors in $H_1$.

In this case, the five vertices $v_2, v_1, x, v_k, v_{k-1}$ belong to a facial cycle $h$ of $G$ (see Fig. \ref{structure-C-beside} (b)).
We claim that both $v_2$ and $v_{k-1}$ have neighbors in $H_2$.
Otherwise, at least one of $v_2$ and $v_{k-1}$ has a neighbor in $H_1$.
If $v_2$ has a neighbor in $H_1$ and $v_{k-1}$ has a neighbor in $H_2$, then the six vertices $v_3, v_2, v_1, x, v_k, v_{k-1}$ lie on a face $h$ of $G$. So $h$ is a hexagon and $C$ has a chord $v_3v_{k-1}$, a contradiction.
For $v_2$ having a neighbor in $H_2$ and $v_{k-1}$ having a neighbor in $H_1$, we can also obtain a chord of $C$, a contradiction.
If both $v_2$ and $v_{k-1}$ have neighbors in $H_1$, then the seven vertices $v_3, v_2, v_1, x, v_{k}, v_{k-1}, v_{k-2}$ belong to a common face $h$ of $G$. This implies that $G$ has a facial cycle of length at least $7$, a contradiction.
So both $v_2$ and $v_{k-1}$ have neighbors in $H_2$,
and $v_2, v_1, x, v_k, v_{k-1}$ lie on a hexagon $h$ of $G$ (see Fig. \ref{structure-C-beside} (b)).
Since $C$ does not have a chord, the path $v_1xv_k$ is a connected component of the induced subgraph $G[C^1]$.

\textbf{Case 3.} $v_1$ has a neighbor in $H_1$ and $v_k$ has a neighbor in $H_2$, or $v_1$ has a neighbor in $H_2$ and $v_k$ has a neighbor in $H_1$.

By the symmetry, it is sufficient to consider that $v_1$ has a neighbor in $H_1$ and $v_k$ has a neighbor in $H_2$.
If $v_2$ has a neighbor in $H_1$, then $v_3$ must have a neighbor in $H_2$, otherwise, $C$ has a chord or $G$ has a facial cycle of length at least seven, a contradiction. As the proof of Case 2, $v_3, v_2, v_1, x, v_k$ lie on a hexagonal facial cycle.
So the path $v_2v_1x$ is a connected component of the induced subgraph $G[C^1]$.
Now, we suppose that $v_2$ has a neighbor in $H_2$.
Then the four vertices $v_k, x, v_1, v_2$ lie on the same face $g$ of $G$.
Since $v_k, x, v_1, v_2$ are all leaves in $S$, $g$ is a pentagon and $v_2, v_k$ have a common neighbor in $H_2$ which is the center of a star in $S$ (see Fig. \ref{structure-C-beside} (c)).
So the path $xv_1$ is a connected component of the induced subgraph $G[C^1]$.

In summary, the component of the induced subgraph $G[C^1]$ which contains $x$ is a path with $2$ or $3$ vertices since $C$ does not have a chord.
\end{proof}
In addition, we have the following Lemma.
\begin{lem}\label{each-comp-C}
Each component of $G[C^i]$ is a path with $2$ or $3$ vertices, $i=1, 2$.
\end{lem}
\begin{figure}[htbp!]
\centering
\includegraphics[height=8.5cm]{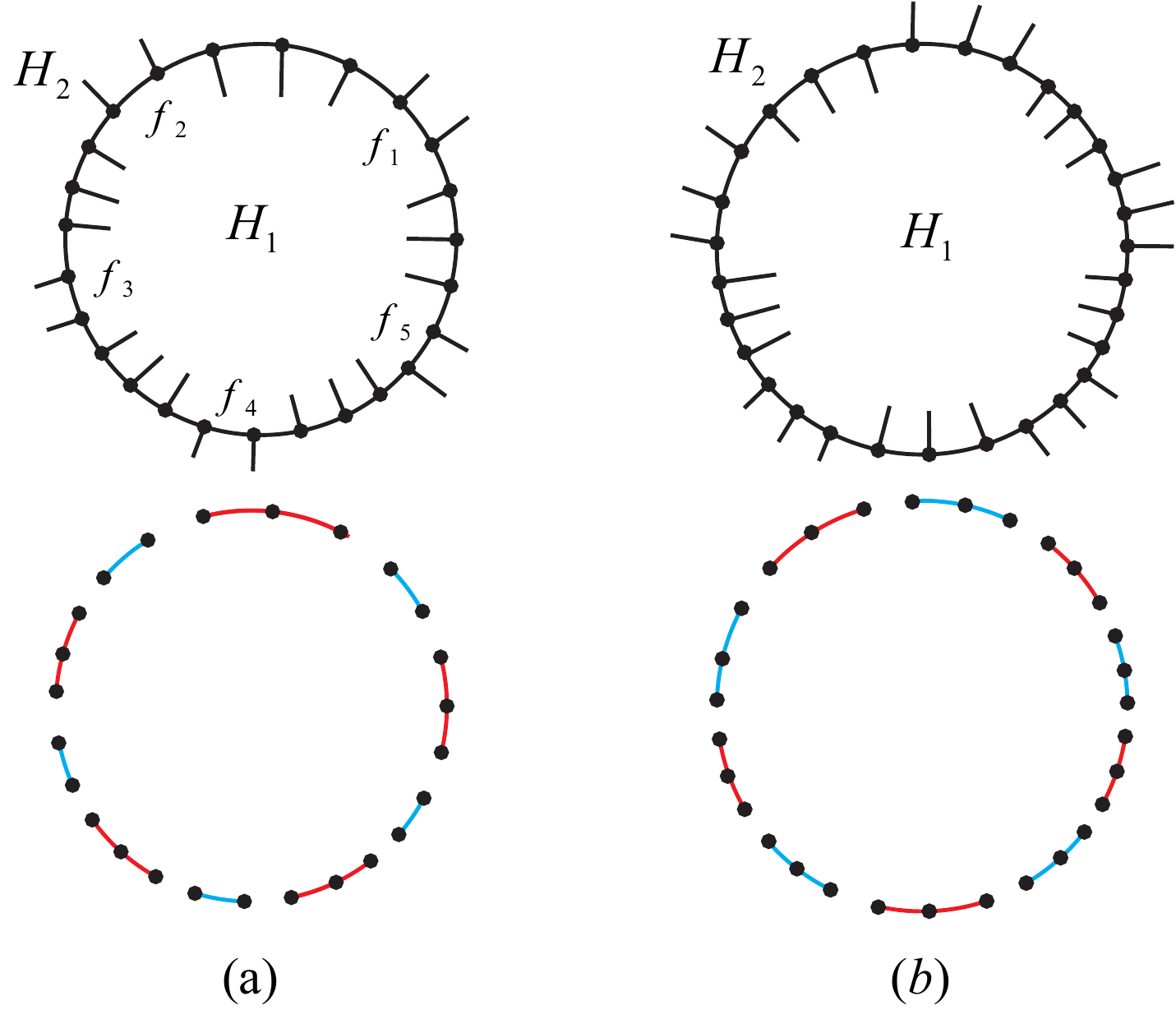}
\caption{\label{structure-C-t}{\small (a) A cycle $C$ of length $25$; (b) A cycle $C$ of length $30$. ($G[C^1]$ is red and $G[C^2]$ is blue.)}}
\end{figure}
\begin{proof}
For any vertex $x$ on $C$, $x$ must have exactly one neighbor in $H_1$ or $H_2$ since $G$ is $3$-regular and $C$ does not have a chord. Without loss of generality, we suppose that $x$ has exactly one neighbor in $H_1$.
By Lemma \ref{structure-C}, the component of the induced subgraph $G[C^1]$ which contains $x$ is a path with $2$ or $3$ vertices .
We note that the choice of $x$ is arbitrary. So the conclusion holds.
\end{proof}
\begin{prop}\label{stucture-C-S}
Let $C=v_0v_1\cdots v_{k-1}$ be a non-facial cycle in $G-C(S)$. $($In the following, the subscript is module $k$$)$.\\
$(i)$ If both $v_i$ and $v_{i+1}$ have neighbors in $H_1$ $($or $H_2$$)$ and $v_{i-1}$ and $v_{i+2}$ have neighbors in $H_2$ $($or $H_1$$)$, then the four vertices $v_{i-1}, v_i, v_{i+1}$ and $v_{i+2}$
lie on a pentagon of $G$.\\
$(ii)$ If $v_i, v_{i+1}, v_{i+2}$ have neighbors in $H_1$ $($or $H_2$$)$ and $v_{i-1}$ and $v_{i+3}$ have neighbors in $H_2$ $($or $H_1$$)$, then the five vertices $v_{i-1}, v_i, v_{i+1}, v_{i+2}$ and $v_{i+3}$
 lie on a hexagon of $G$.\\
$(iii)$ For $j=1, 2$, if both $v_i$ and $v_{i+1}$ have neighbors in $H_j$ $($we denote the two edges incident to $v_i$ and $v_{i+1}$ not lie in $C$ by $e_i$ and $e_{i+1}$, respectively$)$, then
the facial cycle containing both $e_i$ and $e_{i+1}$ is a hexagon, and two antipodal points on this hexagon are centers of two stars in the perfect star packing $S$.
\end{prop}
\begin{proof}
Cases $(i)$ and $(ii)$ can be easily obtained from the proof of the Cases $2$ and $3$ of Lemma \ref{structure-C} (see Fig. \ref{structure-C-beside}).
Since all the vertices on $C$ are leaves in the perfect star packing $S$, the other end of $e_i$ (resp. $e_{i+1}$) which is not on $C$, denoted by $u_{i}$ (resp. $u_{i+1}$), is the center of a star in $S$.
We know that any facial cycle of $G$ is a pentagon or a hexagon. So $u_{i}$ and $u_{i+1}$ are distinct. By Lemmas \ref{1-pentagon} and \ref{1-hexagon}, the facial cycle containing both $e_i$ and $e_{i+1}$ is a hexagon, and $u_{i}$ and $u_{i+1}$ are antipodal points on this hexagon.
\end{proof}
For example, in Fig. \ref{structure-C-t}, except for $f_i, i\in\{1, 2, 3, 4, 5\}$  the other faces sharing edges with $C$ are all hexagons. Moreover, how the vertices on $C$ being covered by $S$ is determined.

We recall that the the union of two graphs $G_1$ and $G_2$ is denoted by $G_1\cup G_2$, which has vertex set $V(G_1)\cup V(G_2)$ and edge set $E(G_1)\cup E(G_2)$.
Let $n_3$ be the number of the components of $G[C^1]\cup G[C^2]$ each of which is isomorphic to a path with $3$ vertices.
Similarly, $n_2$ is the number of the components of $G[C^1]\cup G[C^2]$ each of which is isomorphic to a path with $2$ vertices.
For example, $n_3=n_2=5$ in Fig. \ref{structure-C-t} (a) and $n_3=10$, $n_2=0$ in Fig. \ref{structure-C-t} (b).

\begin{obs}\label{n2n3}
$n_2+n_3$ is even.
\end{obs}
\begin{prop}\label{n2-n3-parity}
Let $S$ be a perfect star packing of fullerene graph $G$ and $C$ a cycle in $G-C(S)$ which is a non-facial cycle of $G$.
Then the length of $C$ is $3n_3+2n_2$, and
has the the same parity with $n_2$ and $n_3$.
\end{prop}
\begin{proof}
Clearly, the length of $C$ is $3n_3+2n_2$ by Lemma \ref{each-comp-C}. So $n_3$ is odd if and only if the length of $C$ is odd.
Since $n_2+n_3$ is even by Observation \ref{n2n3}, the parity of $n_2$ and $n_3$ are same.
Then we are done.
\end{proof}

\begin{thm}\label{even-odd-cycles}
Let $S$ be a perfect star packing of fullerene graph $G$. Then $G-C(S)$ has even number of odd cycles.
\end{thm}
\begin{proof}
If $G-C(S)$ does not have a non-facial cycle of $G$, then any pentagon of $G$ does not have a vertex in $C(S)$ by Corollary \ref{pentaon-and-nonfacialcycle}. So all the vertices on pentagons are leaves in $S$. It implies that $G-C(S)$ has exactly twelve odd cycles, each of which is a pentagon.
Next, we suppose that $G-C(S)$ has a non-facial cycle of $G$, denoted by $C$.

\textbf{Claim 1.} If $C$ is an even cycle, then $G$ has even number of pentagons which share edges with $C$. If $C$ is an odd cycle, then $G$ has odd number of pentagons which share edges with $C$.

By Proposition \ref{stucture-C-S}, the number of pentagons which share edges with $C$ is equal to $n_2$. By Proposition \ref{n2-n3-parity}, $n_2$ and the length of $C$ have the same parity. So the Claim holds.

\textbf{Claim 2.} Any pentagon of $G$ shares edges with at most one non-facial cycle in $G-C(S)$.

Let $P$ be a pentagon of $G$. By Proposition \ref{1-pentagon}, $P$ has at most one vertex which is the center of a star in $S$.
If $P$ does not have a vertex in $C(S)$, then $P$ is a cycle in $G-C(S)$. By Theorem \ref{characterization}, each component of $G-C(S)$ is an induced cycle of $G$. So $P$ does not share edges with any non-facial cycle in $G-C(S)$.
If $P$ has a vertex $x\in C(S)$, then by Corollary \ref{pentaon-and-nonfacialcycle} $P-x$ is a subgraph of a non-facial cycle in $G-C(S)$. So $P$ shares edges with exactly one non-facial cycle in $G-C(S)$.

Now, we consider the following two cases for the non-facial cycles in $G-C(S)$.

\textbf{Case 1.} $G-C(S)$ does not have a non-facial cycle of odd length.

Then any non-facial cycle $C$ in $G-C(S)$ is of even length. By the above Claims, there are even number of pentagons in $G$ such that they share edges with $C$. Since $G$ has exactly twelve pentagons, there are even number of pentagons in $G$ each of which does not share edges with non-facial cycles in $G-C(S)$. These pentagons must be cycles in $G-C(S)$ by Corollary \ref{pentaon-and-nonfacialcycle}. Hence $G-C(S)$ has even number of odd cycles.

\textbf{Case 2.} $G-C(S)$ has some non-facial cycle of odd length.

Suppose that $G-C(S)$ has exactly $k$ non-facial cycles of odd length.
We denote the number of pentagons in $G$ each of which does not share edges with non-facial cycles in $G-C(S)$  by $p$.
These $p$ pentagons must be cycles in $G-C(S)$ by Corollary \ref{pentaon-and-nonfacialcycle}.
So $G-C(S)$ has $p+k$ odd length cycles.
Next, we show that $p$ and $k$ have the same parity.
If $p$ is odd, then $G$ has odd number of pentagons each of which share edges with exactly one non-facial cycle in $G-C(S)$ since $G$ has exactly $12$ pentagons.
By the above Claims, for each even length non-facial cycle in $G-C(S)$, $G$ has even number of pentagons which share edges with the cycle, and for each odd length non-facial cycle in $G-C(S)$, $G$ has odd number of pentagons which share edges with the cycle.
So $G-C(S)$ has odd number of non-facial cycles of odd length.
This means that $k$ is odd.
For $p$ being even, we can similarly show that $k$ is even.
So $k$ and $p$ have the same parity and $p+k$ is even.
\end{proof}

Clearly, for a fullerene graph $G$ with a perfect star packing, its order must be divisible by $4$. So the order of $G$ is $8k$ or $8k+4$ for some positive integer $k$.
Now, we can obtain the following main theorem which illustrate that the order of $G$ can not be $8k+4$.
\begin{thm}\label{main-thm}
If fullerene graph $G$ has a perfect star packing, then the order of $G$ is divisible by $8$.
\end{thm}
\begin{proof}
We suppose that $S$ is a perfect star packing of $G$ and $\mathcal{C}_o$ and $\mathcal{C}_e$ are the collections of all the odd cycles and even cycles in $G-C(S)$, respectively. Then we have the following equation.
\begin{equation}
\begin{split}
|V(G)| &=|C(S)|+\sum_{C\in\mathcal{C}_o}|C|+\sum_{C\in\mathcal{C}_e}|C|\\
& = \frac{|V(G)|}{4}+\sum_{C\in\mathcal{C}_o}|C|+even.\\
\end{split}
\end{equation}
By Theorem \ref{even-odd-cycles}, $\mathcal{C}_o$ has even number of elements. Combine the above equation, we know that $\frac{|V(G)|}{4}\times3$ is even. Hence $\frac{|V(G)|}{4}$ is even, that is, the order of $G$ is divisible by $8$.
\end{proof}
This theorem is equivalent to the following corollary.
\begin{cor}
A fullerene graph with order $8n+4$ does not have a perfect star packing.
\end{cor}

We recall that
a \emph{dominating set} of a graph $G$ is a set of vertices $D$ such that each vertex in $V(G)-D$ is adjacent to a vertex in $D$.
Moreover, if each vertex in $V(G)-D$ is adjacent to exactly one vertex in $D$ and $D$ is an independent vertex set, then $D$ is called \emph{efficient}.
The problem of determining the existence of efficient dominating sets in some families of graphs was first investigated by Biggs \cite{Biggs} and Kratochvil \cite{Kratochvil}.
Later Livingston and Stout \cite{per_dominating} studied the existence and construction of efficient dominating sets in families of graphs arising from the interconnection networks of parallel computers. The problem of finding an efficient dominating set,
however, is algorithmically hard \cite{Alg_hard}.
For more results and some historical background regarding efficient dominating set, we refer the reader to \cite{Worst_case_effidom, effi_cayley, dom_ver_tran, circulane_effidom} etc.

From the definitions of the efficient dominating set and the perfect star packing of a fullerene graph $G$,
the following proposition is a natural result.
 \begin{prop}\label{pspacking-effidominating}
 A fullerene graph $G$ with $n$ vertices has a perfect star packing if and only if $G$ has an efficient dominating set of cardinality $\frac{n}{4}$.
 \end{prop}
 Combine Theorem \ref{main-thm} and Proposition \ref{pspacking-effidominating}, we get the following theorem.
\begin{thm}
The order of a fullerene graph with an efficient dominating set is $8n$.
\end{thm}

\section{Some other conclusions}
Do\v{s}li\'{c} et. al. gave the following spectral necessary condition for the existence of a perfect star packing in a fullerene graph.
\begin{prop}[\cite{Doslic}]
If a fullerene graph $G$ has a perfect star packing, then $-1$ must be an eigenvalue of the adjacency matrix of $G$.
\end{prop}
The proof of this Theorem can be translate to a simple $r$-regular graph. Here for completeness, we prove as follows.
For the definition of eigenvalues of the adjacency matrix of a graph, we refer the reader to \cite{eigenvalue}.
\begin{thm}
If a simple $r$-regular graph $G$ has a perfect $K_{1,r}$-packing $S$, then $-1$ must be an eigenvalue of the adjacency matrix of $G$.
\end{thm}
\begin{proof}
Let $C(S)$ be the set of centers of stars $K_{1,r}$ in $S$. We define the characteristic vector $\overrightarrow{c}\in \mathbb{R}^{|V(G)|}$ of $C(S)$ as follows: $c_i=1$ if $i\in C(S)$, otherwise $c_i=0$.
Set $\overrightarrow{u}$ be the vector of all ones.
For the adjacency matrix $A$ of $G$, we have $A\overrightarrow{u}=r\overrightarrow{u}$ since $G$ is $r$-regular.
Let $\overrightarrow{w}=\overrightarrow{u}-(r+1)\overrightarrow{c}$. As $A\overrightarrow{c}=\overrightarrow{u}-\overrightarrow{c}$, we have
\begin{equation}
A\overrightarrow{w}=A\overrightarrow{u}-(r+1)A\overrightarrow{c}=r\overrightarrow{u}-(r+1)\overrightarrow{u}+(r+1)\overrightarrow{c}
=(r+1)\overrightarrow{c}-\overrightarrow{u}=-\overrightarrow{w}
\end{equation}
This implies that $-1$ is an eigenvalue of $A$.
\end{proof}
A perfect star packing $S$ of a fullerene graph $G$ is of \emph{type $P0$} if no center of a star in $S$ is on a pentagon of $G$. For such perfect star packing, the following corollary holds.
\begin{cor}
If fullerene graph $G$ has a perfect star packing $S$ of type $P0$, then $G-C(S)$ does not have a non-facial cycle of odd length.
\end{cor}
\begin{proof}
By the contrary, we suppose that $G-C(S)$ has a non-facial cycle $C$ of odd length.
By the Claim $1$ of Theorem \ref{even-odd-cycles}, $G$ has a pentagon $P$ which share edges with $C$.
This implies that $P$ contains the center of a star in $S$.
This contradicts that $S$ is of type $P0$.
So $G-C(S)$ does not have a non-facial cycle of odd length.
\end{proof}
In the above Corollary, we note that $G-C(S)$ may have non-facial cycles of even lengths (see Fig. \ref{examples}, the blue cycle in $C_{120}$).

Now, we point out the error of the Theorem $14$ in \cite{Doslic}.
\begin{thm}[\cite{Doslic}]\label{type-P0-Doslic}
A fullerene graph on $8n$ vertices has a perfect star packing of type $P0$ if and only if it arises from some other fullerene via the chamfer transformation.
\end{thm}
\begin{figure}[htbp!]
\centering
\includegraphics[height=13.5cm]{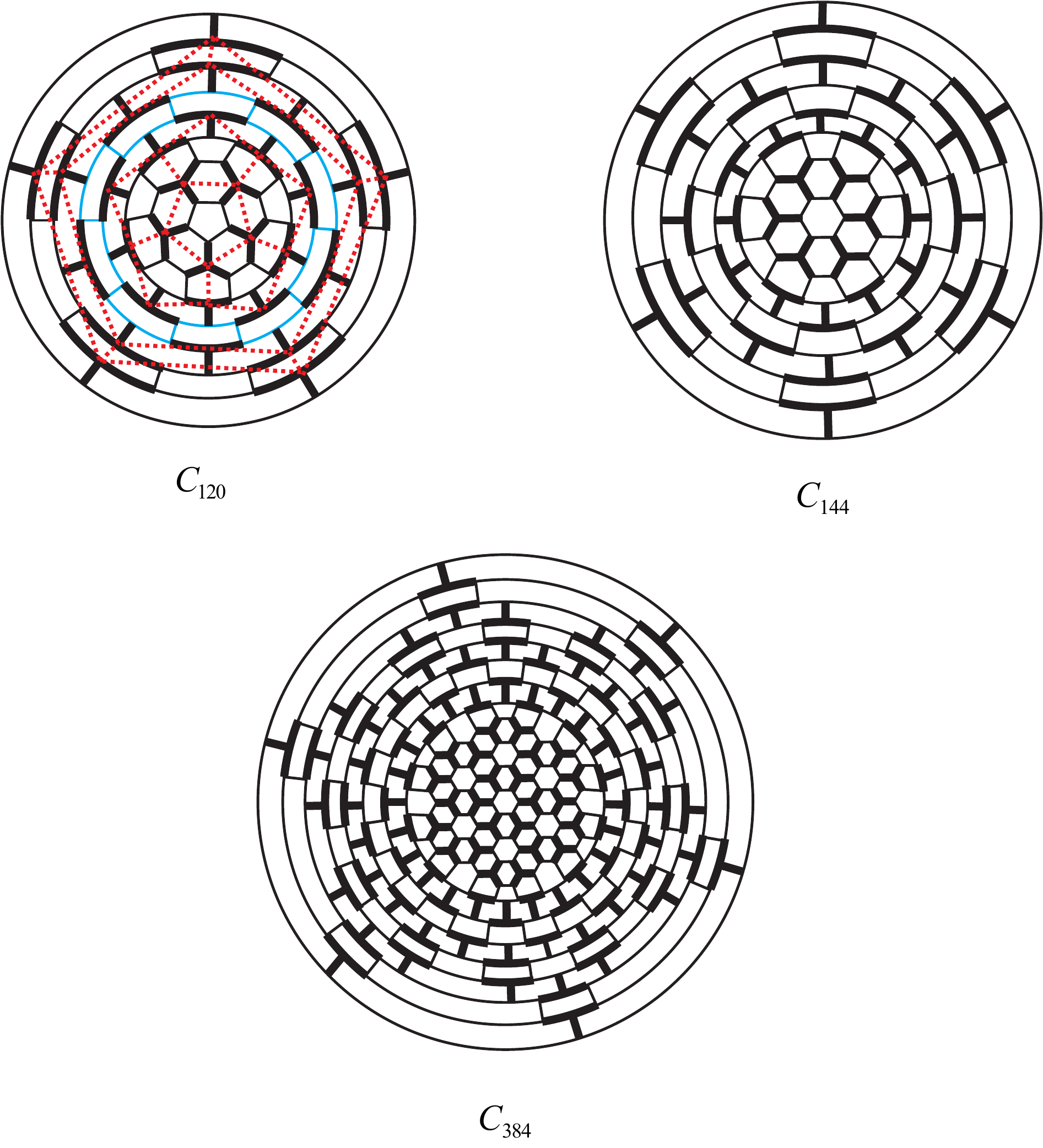}
\caption{\label{examples}{\small Each of $C_{120}, C_{144}, C_{384}$ has a unique perfect star packing of type $P0$ which is depicted in bold edges.}}
\end{figure}
In the proof of the necessity of this Theorem, there exist the following problem.
Take a fullerene graph with a perfect star packing of type $P0$. All star centers lie on vertices shared by three hexagons. When we connect the centers of stars lying on the same hexagons,
the resulting graph is planar, but does not have to be $3$-regular, $3$-connected and have only pentagonal and hexagonal faces.
For example, it is easy to check that
each of the fullerene graphs $C_{120}, C_{144}, C_{384}$ (see Fig. \ref{examples}) has a unique perfect star packing of type $P0$. When we connect
 the centers of stars lying on the same hexagons, the resulting graph (the red dashed line
in Fig. \ref{examples} is the resulting graph for $C_{120}$, and here we omit the resulting graphs for $C_{144}$ and $C_{384}$) is planar and is not connected.
In fact, the three fullerene graphs $C_{120}, C_{144}$ and $C_{384}$ as depicted in Fig. \ref{examples}
cannot arise from some other fullerene via the chamfer transformation.
So the necessity of this theorem does not hold, however, its sufficiency is right. It can be corrected as follows.
\begin{thm}
A fullerene graph that arises from some other fullerene via the chamfer transformation must have a perfect star packing of type $P0$.
\end{thm}

From paper \cite{Doslic},
we know that fullerenes with two pentagons sharing an edge can not have a perfect star packing of type $P0$ since the edge shared by the two pentagons cannot lie in any star.
Next we list some other forbidden configurations whose presence in a fullerene graph precludes the existence of a perfect star packing of type $P0$.
\begin{figure}[htbp!]
\centering
\includegraphics[height=2.8cm]{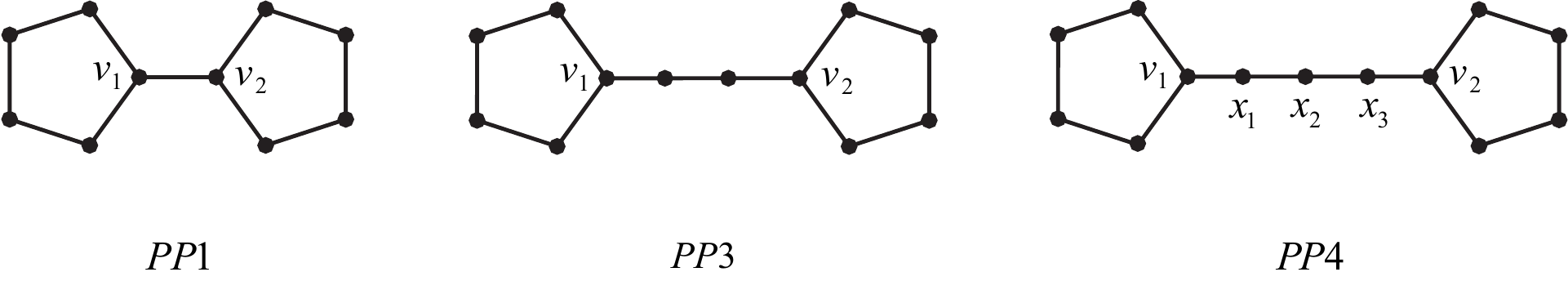}
\caption{\label{PP}{\small Three forbidden configurations.}}
\end{figure}
\begin{prop}
If a fullerene graph $G$ contains a subgraph $PP1, PP3$ or $PP4$ \emph{(}see Fig. \ref{PP}\emph{)}, then it cannot have a perfect star packing of type $P0$.
\end{prop}
\begin{proof}
By the contrary, we suppose that $G$ has a perfect star packing of type $P0$, denoted by $S$.
Clearly, the vertices $v_1$ and $v_2$ (see Fig. \ref{PP}) are leaves in $S$.
If $PP4$ is a subgraph of $G$, then $x_1$ is the center of a star in $S$ since all vertices on a pentagon are leaves in $S$.
So $x_2$ is a leaf in $S$.
By Remark \ref{remark1} 2, the neighbor $x_3$ of $x_2$ is also a leaf in $S$.
This implies that all the three neighbors of $v_2$ are leaves in $S$, a contradiction.
For subgraphs $PP1$ and $PP3$, we can similarly show that $v_2$ has all its three neighbors being leaves in $S$, a contradiction.
\end{proof}

\section{Acknowledgments}
 I would like to sincerely thank Wuyang Sun for his careful reading and valuable comments and suggestions.


\end{document}